\documentclass[11pt]{article}
\usepackage[utf8]{inputenc}
\usepackage{xcolor,color}
\usepackage{url}
\usepackage{authblk}

\usepackage{hyperref}
\hypersetup{colorlinks,linkcolor=black,filecolor=black,urlcolor=black,citecolor=black}

\usepackage{pstricks}
\usepackage{pdftricks}
\usepackage{amsmath,amssymb,amsthm,amsfonts}

\theoremstyle{plain}% default
\newtheorem{theorem}{Theorem}%[section]
\newtheorem{observation}[theorem]{Observation}

\newtheorem{lemma}[theorem]{Lemma}
\newtheorem{proposition}[theorem]{Proposition}

\theoremstyle{definition}
\newtheorem{definition}[theorem]{Definition}

\usepackage[shortlabels]{enumitem} % for better lists

\newcommand{\set}[1]{\left\{#1\right\}}
\newcommand{\abs}[1]{\left|#1\right|}

\begin{psinputs}
\usepackage{pstricks}
\usepackage{pstricks-add}
\usepackage{color}
\usepackage{pstcol}
\usepackage{pst-plot}
\usepackage{pst-tree}
\usepackage{pst-eps}
\usepackage{multido}
\usepackage{pst-node}
\usepackage{pst-eps}
\SpecialCoor
\newcommand{\vx}[2]{\cnode*(#1){1.2pt}{#2}}

\newcommand{\edg}[2]{\ncline{#1}{#2}}

\end{psinputs}

\newdimen\omsq  \omsq=10pt
\newdimen\omrule    \omrule=1pt
\newdimen\omint

\newif\ifvth    \newif\ifhth    \newif\ifomblank
\def\OMINO#1{%
    \vthtrue \hthtrue
    \vbox{ \offinterlineskip\parindent=0pt \OM#1\relax\vskip1pt}}

\def\OM#1{%
    \omint=\omsq    \advance\omint-\omrule
    \ifx\relax#1%
    \else
      \ifx\\#1 \newline\null \hthtrue \ifvth\vthfalse\else\vskip-\omrule\vthtrue\fi
      \else%
        \ifx .#1\hskip\ifhth \omrule\else \omint\fi
        \else%
          \ifx +#1\def\colour{black}\fi%
          \ifx -#1\def\colour{black}\fi%
          \ifx |#1\def\colour{black}\fi%
          \ifx 2#1\def\colour{lightgray}\fi%
          \ifx 3#1\def\colour{gray}\fi%
          \ifx 4#1\def\colour{darkgray}\fi%
          \ifx @#1\def\colour{black}\fi%
          \ifx r#1\def\colour{red}\fi%
          \ifx g#1\def\colour{green}\fi%
          \ifx l#1\def\colour{lime}\fi%
          \ifx o#1\def\colour{olive}\fi%
          \ifx O#1\def\colour{orange}\fi%
          \ifx b#1\def\colour{blue}\fi%
          \ifx t#1\def\colour{teal}\fi%
          \ifx y#1\def\colour{yellow}\fi%
          \ifx m#1\def\colour{magenta}\fi%
          \ifx c#1\def\colour{cyan}\fi%
          \ifx p#1\def\colour{pink}\fi%
          \ifx P#1\def\colour{purple}\fi%
          \ifx w#1\def\colour{white}\fi%
          \textcolor{\colour}{\rule{\ifhth\omrule\else\omsq\fi}{\ifvth\omrule\else\omsq\fi}}%
          \ifhth\else\hskip -\omrule\fi%
        \fi%
        \ifhth\hthfalse\else\hthtrue\fi%
      \fi%
    \expandafter\OM%
    \fi}

\makeatother

\newcommand{\ep}{\varepsilon}

\title{Subexponential mixing for partition chains on grid-like graphs}
\author{Alan Frieze\thanks{Research Supported in part by NSF grant DMS1952285. email: frieze@cmu.edu}}
\author{Wesley Pegden\thanks{Research Supported in part by NSF grant DMS1700365. email: wes@math.cmu.edu}}
\affil{Department of Mathematical Sciences, Carnegie Mellon University}

\newcommand{\cI}{\mathcal{I}}
\newcommand{\cX}{\mathcal{X}}

\newcommand{\R}{\mathbb{R}}

\newcommand{\cM}{\mathcal{M}}
\newcommand{\cG}{\mathcal{G}}

\newcommand{\cE}{\mathcal{E}}

\newcommand{\cP}{\mathcal{P}}

\newcommand{\band}{\mathrm{band}}
\newcommand{\poly}{\mathrm{poly}}
\begin{document}
\maketitle
\begin{abstract}
We consider the problem of generating uniformly random partitions of the vertex set of a graph such that every piece induces a connected subgraph.  For the case where we want to have partitions with linearly many pieces of bounded size, we obtain approximate sampling algorithms based on Glauber dynamics which are fixed-parameter tractable with respect to the bandwidth of $G$, with simple-exponential dependence on the bandwidth.  For example, for rectangles of constant or logarithmic width this gives polynomial-time sampling algorithms.  More generally, this gives sub-exponential algorithms for bounded-degree graphs without large expander subgraphs (for example, we obtain $O(2^{\sqrt n})$ time algorithms for square grids).

In the case where we instead want partitions with a small number of pieces of linear size, we show that Glauber dynamics can have exponential mixing time, even just for the case of 2 pieces, and even for 2-connected subgraphs of the grid with bounded bandwidth.
\end{abstract}

\section{Introduction}
In this paper we consider the mixing time of Glauber dyanmics for uniform sampling of partitions of the vertex-set of a bounded-degree graph into $q$ subsets which each induce connected subgraphs of $G$.  Our positive results in this paper focus on the case where $q$ is linear in the number of vertices of $G$ and each subset in the partition is constrained to have bounded size $\leq B$; in particular, this generalizes the setting of the well-studied monomer-dimer model \cite{danamonomer,jerrummonomer} where $B=2$.

The problem of randomly partitioning a graph into connected pieces arises naturally in the context of evaluating political districtings, where Markov chain methods have been applied and developed over the past decade.  (For this application, our positive results correspond to the case where there are many small districts; in the U.S. context, one might imagine the case of the state legislative districting of a state, rather than its districting into US Congressional districts.)  While there are rigorous statistical approaches based on Markov chains which can avoid the sampling problem when the only goal is outlier detection \cite{outliers,twopaths}, the benefits of sampling have led to the development of a large range of Markov chains intended to generate good samples from partition spaces \cite{recomb,kosuke2020,herschlag2020quantifying,graves2017,duke2019mergesplit,duke2020multi,duke2021metropolized}.  The mixing properties of these Markov chains, however, have only been supported by heuristic evidence, and there is still a dearth of results allowing rigorous approximate sampling from specified target distributions on such partitions.  

\bigskip

%; within this context, our methodology is quite general, allowing analysis of cases where we are interested in sampling partitions with pieces of size $\leq B$, or, say, partitions of a grid using only pieces of size 4,5,7, or 11.  
The mixing times we prove will depend on the \emph{bandwidth} of $G$, which is the minimum over labelings $\sigma$ of the vertices of $G$ by distinct numbers $1,\dots,n$ of the maximum difference $\sigma(u)-\sigma(v)$ of any adjacent pair $u\sim v$ of vertices:
\begin{equation}
\band(G)=\min_{\sigma} \max_{u\sim v} |\sigma(u)-\sigma(v)|.
\end{equation}
(Throughout, we will use $\sigma$ to denote a labeling realizing this minimum.) In particular, when $G$ is a $k\times \ell$ grid with $n=k\ell$ vertices we have $\band(G)=\min(k,\ell)$ \cite{rectanglebandwidth}, and will obtain running times of order $\poly(n)2^{O(\min(k,\ell))}$.

\bigskip
In this paper, a contiguous $q$-partition of $G$ is a partition of the vertex set into $q$ (not necessarily nonempty) pieces, each inducing a connected subgraph of $G$, % a $\cB$-constrained contiguous $q$-coloring is one in which the size of every piece belongs to the set $\cB$,
and a $B$-bounded contiguous $q$-partition is a contiguous $q$-partition in each every piece has size $\leq B$.  %$\cB$-constrained contiguous $q$-coloring for $\cB=\{0,1,\dots,B\}$.

In general, it can be NP-hard to determine whether a given graph admits a partition into connected pieces of prescribed sizes, even for planar graphs of bounded degree \cite{planarpartition}.  But things change if we allow some slack in the size constraints.   For example, we will see that one can efficiently partition any graph of maximum degree $\Delta$ into connected pieces of sizes between $B/(\Delta-1)$ and $B$, for any $B\geq 1$.  The Hamiltonicity of grid graphs means they can be partitioned into connected pieces of any sizes summing to the total number of vertices; i.e., essentially no slack is needed at all.  Moreover, in both of these cases, the partitioning can also be done locally, without global modifications to the graph.  The following definition captures these properties:
\begin{definition}\label{def:partitionable}
A family of graphs $\cG$ is $(B,\alpha)$-\emph{partitionable} if for any $\ep>0$ and $q>(1+\ep)n/(\alpha B)$, there's a $K$ sufficiently large and a $\delta>0$ such for any sufficiently large $G\in \cG$:
\begin{itemize}
    \item There is a polynomial time algorithm to find a $B$-bounded contiguous $q$-partition of $G$, and
    \item Given any $B$-bounded contiguous-$q$-partition $\omega$ of $G$, there are $\delta|V(G)|$ distinct $B$-bounded contiguous partitions $\omega_i$ of $G$, using fewer than $q$ partition classes, which each agree with $\omega$ except on a connected subgraph $H_i$ of $G$ with at most $K$ vertices.%, and such that $\omega_i$ induces a $B$-bounded contiguous partition of $H_i$.
    \end{itemize}
%For the set $\cB$ with $B=\max(\cB)$, a family of graphs $\cG$ is $(\cB,\alpha)$-\emph{partitionable} if for any $\ep>0$, there's a $K$ sufficiently large and a $\delta>0$ such for any sufficiently large $G\in \cG$ and any $B$-bounded contiguous-$\kappa$-partition $\omega$ of $G$ with $\kappa> (1+\ep)n/(\alpha B)$, there are $\delta|V(G)|$ distinct $B$-bounded contiguous partitions $\omega_i$ of $G$, using fewer than $\kappa$ partition classes, which each agree with $c$ except on a connected subgraph $H_i$ of $G$ with at most $K$ vertices, and such that $\omega_i$ induces a $\cB$-constrained contiguous partition of $H_i$.
\end{definition}
%\begin{definition}
%For the integer $B$, a family of graphs $\cG$ is $(B,\alpha)$-\emph{partitionable} if it is $(\cB,\alpha)$-partitionable for $\cB=\{0,1,\dots,B\}$.
%\end{definition}
Essentially, the second part of Definition \ref{def:partitionable} requires that any $B$-bounded contiguous partition using sufficiently many nonempty classes can be locally modified to use fewer nonempty color classes. Note that no infinite family of graphs can be $(B,\alpha)$-partitionable for any value $\alpha>1$. On the other hand, we will prove the following result about bounded-degree graphs:
\begin{proposition}\label{p.connectedgraphs}
For any $B$, %and
%\[
%\cB=\{\lceil\tfrac{B-1}{\Delta-1}\rceil,\lceil\tfrac{B-1}{\Delta-1}\rceil+1,\dots,B-1,B\},
%\]
connected graphs of maximum degree $\Delta$ are $(B,\frac{1}{\Delta-1}\frac{B-1}{B})$-partitionable.%  In particular, they are  $(B,\frac{1}{\Delta-1}\frac{B-1}{B})$-partitionable.
\end{proposition}
\noindent Proposition \ref{p.connectedgraphs} cannot be improved in general, but, for example, grids satisfy a strong version of Definition \ref{def:partitionable}:
\begin{proposition}\label{p.grids}
Rectangular grids are $(B,1)$-partitionable for all $B$.%$\cB$ with $|\cB|\geq 2$; in particular they are $(B,1)$-partitionable for all $B$.
\end{proposition}

\noindent Our main result is the following:
\begin{theorem}\label{t.part}If $\cG$ is a $(B,\alpha)$-partitionable family of graphs of bounded maximum degree, then for any 
\[
q\geq (1+\ep)\frac{n}{\alpha B},
\]
Glauber dynamics for the space of $B$-bounded contiguous $q$-partitions of $G$ has mixing time
\[
\leq \poly(n)2^{O(\band(G))}\log(\tfrac 1 \delta).
\]
\end{theorem}
We define the Glauber dyanmics for this state space in Section \ref{s.glauber}.
  % For example, in Section \ref{s.cB} we will discuss sampling algorithms for $k\times \ell$ grids which use partition classes whose sizes are bounded both from above and below.:
%\begin{theorem}
%For any set $\cB$, if $\cG$ is a $(\cB,\alpha)$-partitionable family of graphs of bounded maximum degree, then there is an algorithm $A$ such that for any 
%\[
%q\geq (1+\ep)\frac{n}{\alpha \max(\cB)}
%\]
%and any $\delta>0$, $A(G)$ generates $\cB$-constrained contiguous $q$-partitions of $G$, which samples from a distribution whose total variation distance from uniform is $\leq \delta$, with running time
%\[
%\poly(n)2^{O(\band(G))\log n}\log(\tfrac 1 \delta).
%\]
%\end{theorem}
We can also accomodate reasonable nonuniformity in the target distribution.  For example, letting $\kappa(\omega)$ denote the number of nonempty classes in a partition $\omega$ of the vertex set of $G$, one natural case is the following:
\begin{theorem}\label{t.color}If $\cG$ is a $(B,\alpha)$-partitionable family of graphs of bounded maximum degree, then for for any 
\[
q\geq (1+\ep)\frac{n}{\alpha B},
\]
Glauber dynamics for the space of $B$-bounded contiguous $q$-partitions of $G$ has mixing time
\[
\leq \poly(n)2^{O(\band(G))}\log(\tfrac 1 \delta),
\]
where $\pi(\omega)$ for a partition $\omega$ is proportional to $\frac{q!}{(q-\kappa(\omega))!}$.
\end{theorem}
Note that this weighting corresponds to the case where we select a $q$-partition by selecting a random $q$-coloring and then ignoring the colors.

The mixing time given by all these results depend exponentially just on the bandwidth of $G$; they give algorithms for approximate sampling that are fixed-parameter-tractable for the bandwidth, and in particular we have subexponential (i.e., $2^{o(n)}$) algorithms any time $\cG$ is a family of graphs with sublinear bandwidth.  From a result of B{\"o}ttcher, Pruessmann, Taraz, and W{\"u}rfl \cite{characterizingbandwidth}, this is equivalent to requiring graphs in $\cG$ to have sublinear treewidth, or requiring that graphs in $\cG$ have no linear-size subgraphs with good expansion.  Thus our results give subexponential time algorithms in all these cases (and for all planar graphs, by the planar separator theorem).  Note, for example, that a square grid graph on $n$ vertices has bandwidth $\sqrt n$, as can be realized by the row-by-row ordering if its vertices.

Najt, Deford, and Solomon analyzed the worst-case complexity of sampling connected $q$-partitions in the case where $q$ is bounded and thus the typical piece-size is linear \cite{najt2019complexity}.  They show, for example that it is NP-hard to sample connected 2-partitions for planar graphs.  They also show that sampling connected $q$-partitions (for bounded $q$) is fixed-parameter tractable in the treewidth, although with a tower function dependence on the treewidth.  On Markov Chains, they show that there are families of graphs (even among, say, 3-connected planar graphs of bounded degree) for which Glauber dynamics is torpidly mixing for the case of connected $2$-partitions.  In Section \ref{s.torpid}, we will show that Glauber dynamics can even be torpidly mixing for $q=2$ on bounded-bandwidth subgraphs of the grid.  In particular, we will give an example of a family of graphs which are simply the union of two overlapping grid graphs of bounded bandwidth, with exponential mixing time.

\section{Proofs of positive results}
\subsection{Markov chains for partitions}\label{s.glauber}
In this section we define the Glauber dynamics for connected partitions.  For a fixed connected graph $G$ on $n$ vertices, we let $\Omega_{q,B,G}$ consist of all $B$-bounded contiguous-$q$-partitions of $G$.

We consider a Markov Chain $\cM$ on $\Omega_{q,B,G}$ defined by the following transition procedure from a state $\omega\in \Omega_{q,B,G}$:
\begin{enumerate}[(a)]
    \item Choose a uniformly random vertex $v\in G$;
    \item \label{s.uniform} Choose a uniformly random set among all sets in the partition $\omega$;
    \item Remove $v$ from its current set and make it a member of the randomly chosen set, if this move would result in a partition belonging to the state space $\Omega_{q,B,G}$.
\end{enumerate}

When step \ref{s.uniform} is carried out, there is a question of how one handles sets which occur multiple times in a fixed partition class $\omega$---do we consider $\omega$ to be a set of partition classes or a multiset? The only set which can occur multiple times as a partition class is the empty set, which will appear $q-\kappa$ times, where $\kappa$ is the number of nonempty color classes in $\omega$.  In the case where we include the empty-set with multiplicity in the uniform random selection in Step \ref{s.uniform}, the stationary distribution of $\cM$ is not uniform on partitions but on colorings; e.g., partitions with more nonempty classes are more likely.  In the case where we include the empty-set without multiplicity, the uniform distribution on partitions is stationary.  Our methods are not sensitive to the differences between these two choices of the chain.  The first chain will be the underpinning of Theorem \ref{t.color}, while the second underpins Theorem \ref{t.part}.   Note that in the first case, we can consider $\Omega$ to be the set of $q$-colorings instead of $q$-partitions, so that we have $\pi(\omega)=\frac{1}{|\Omega|}$ in all cases.

In fact, $\cM$ can fail to be irreducible on the state space $\Omega_{q,B,G}$, even when $q=Cn$ and $B$ is large and $G$ is a grid graph (Figure \ref{fig:irreducible}).   But we will define $\Omega^K_{q,B,G}$ to be those $q$-partitions $\omega$ of $G$ which agree with some $\omega'$ in $\Omega_{q,B,G}$ on all but at most $K$ vertices.  This is the chain we will show is rapidly mixing using the method of canonical paths.

\begin{figure}
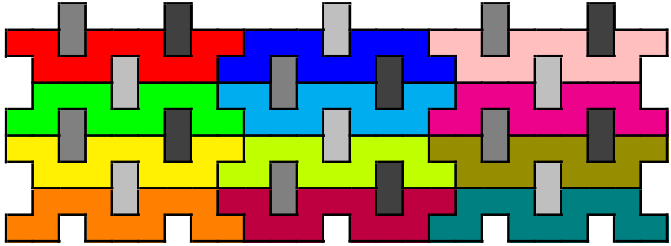

    \centering
\OMINO{
....+-+.....+-+.........+-+.........+-+.....+-+....\\
....|3|.....|4|.........|2|.........|3|.....|4|....\\
+---+.+-----+.+---+-----+.+-----+---+.+-----+.+---+\\
|r.r|3|r.r.r|4|r.r|b.b.b|2|b.b.b|p.p|3|p.p.p|4|p.p|\\
+-+.+-+.+-+.+-+.+-+.+-+.+-+.+-+.+-+.+-+.+-+.+-+.+-+\\
.w|r.r.r|2|r.r.r|b.b|3|b.b.b|4|b.b|p.p.p|2|p.p.p|w.\\
..+-----+.+-----+---+.+-----+.+---+-----+.+-----|w.\\
.w|g.g.g|2|g.g.g|c.c|3|c.c.c|4|c.c|m.m.m|2|m.m.m|w.\\
+-+.+-+.+-+.+-+.+-+.+-+.+-+.+-+.+-+.+-+.+-+.+-+.+-+\\
|g.g|3|g.g.g|4|g.g|c.c.c|2|c.c.c|m.m|3|m.m.m|4|m.m|\\
+---+.+-----+.+---+-----+.+-----+---+.+-----+.+---+\\
|y.y|3|y.y.y|4|y.y|l.l.l|2|l.l.l|o.o|3|o.o.o|4|o.o|\\
+-+.+-+.+-+.+-+.+-+.+-+.+-+.+-+.+-+.+-+.+-+.+-+.+-+\\
.w|y.y.y|2|y.y.y|l.l|3|l.l.l|4|l.l|o.o.o|2|o.o.o|w.\\
..+-----+.+-----+---+.+-----+.+---+-----+.+-----|w.\\
.w|O.O.O|2|O.O.O|P.P|3|P.P.P|4|P.P|t.t.t|2|t.t.t|w.\\
+-+.+-+.+-+.+-+.+-+.+-+.+-+.+-+.+-+.+-+.+-+.+-+.+-+\\
|O.O|w|O.O.O|w|O.O|P.P.P|w|P.P.P|t.t|w|t.t.t|w|t.t|\\
+---+.+-----+.+---+-----+.+-----+---+.+-----+.+---+\\
}
    \caption{When tiled, this configuration has maximum partition class size $13$, average class size $32/5=6.4$, but is rigid (no Glauber dynamics transitions are possible) if $B=13$ unless there are unused colors.  Any change to existing pieces which preserves continuity either disconnects one of the large pieces, or increases the size of a large piece beyond the threshold $B=13$.}
    \label{fig:irreducible}
    \end{figure}

\bigskip

\subsection{Three Lemmas}A crucial ingredient of our proof is the so-called \emph{mountain-climbing problem} \cite{mountain52,mountain89}.  Roughly speaking, this problem asserts that two mountain climbers, each beginning at sea level, can traverse a 2-dimensional mountain range in such a way that they meet at the summit, and at all times are at equal altitudes (Figure \ref{fig:mountain}).  Note that, in general, this requires the mountain climbers to change directions (and sometimes walk away from the summit), and that quadratically many such changes can be needed, in the number of local of extrema of the mountain range.

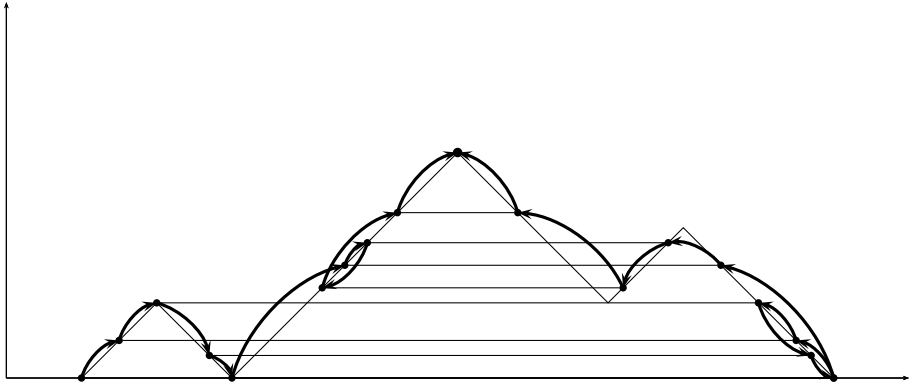
\begin{figure}[t]
  \begin{center}
\begin{pdfpic}
    \psset{unit=1cm,arrowscale=1}
    \begin{pspicture}(-1,0)(10,5)
      \psaxes[labels=none,ticks=none,linewidth=.1pt]{->}(-1,0)(11,5)
\psline[linewidth=.4pt](0,0)(1,1)(2,0)(5,3)(7,1)(8,2)(10,0)

\pnode(0,0){A}\pnode(10,0){a}
\pnode(.5,.5){B}\pnode(9.5,.5){b}
\pnode(1,1){C}\pnode(9,1){c}
\pnode(1.7,.3){D}\pnode(9.7,.3){d}
\pnode(2,0){E}\pnode(10,0){e}
\pnode(3.5,1.5){F}\pnode(8.5,1.5){f}
\pnode(3.8,1.8){G}\pnode(7.8,1.8){g}
\pnode(3.2,1.2){H}\pnode(7.2,1.2){h}
\pnode(4.2,2.2){I}\pnode(5.8,2.2){i}
\pnode(5,3){Jj}

\psline[linewidth=.4pt,arrows=*-*](A)(a)
\psline[linewidth=.4pt,arrows=*-*](B)(b)
\psline[linewidth=.4pt,arrows=*-*](C)(c)
\psline[linewidth=.4pt,arrows=*-*](D)(d)
\psline[linewidth=.4pt,arrows=*-*](E)(e)
\psline[linewidth=.4pt,arrows=*-*](F)(f)
\psline[linewidth=.4pt,arrows=*-*](G)(g)
\psline[linewidth=.4pt,arrows=*-*](H)(h)
\psline[linewidth=.4pt,arrows=*-*](I)(i)

\psdot(Jj)

\ncarc[arcangleA=30,arcangleB=30,linewidth=1.2pt]{->}{A}{B}
\ncarc[arcangleA=30,arcangleB=30,linewidth=1.2pt]{->}{B}{C}
\ncarc[arcangleA=30,arcangleB=30,linewidth=1.2pt]{->}{C}{D}
\ncarc[arcangleA=30,arcangleB=30,linewidth=1.2pt]{->}{D}{E}
\ncarc[arcangleA=30,arcangleB=30,linewidth=1.2pt]{->}{E}{F}
\ncarc[arcangleA=30,arcangleB=30,linewidth=1.2pt]{->}{F}{G}
\ncarc[arcangleA=30,arcangleB=30,linewidth=1.2pt]{->}{G}{H}
\ncarc[arcangleA=30,arcangleB=30,linewidth=1.2pt]{->}{H}{I}
\ncarc[arcangleA=30,arcangleB=30,linewidth=1.2pt]{->}{I}{Jj}

\ncarc[arcangleA=-30,arcangleB=-30,linewidth=1.2pt]{->}{a}{b}
\ncarc[arcangleA=-30,arcangleB=-30,linewidth=1.2pt]{->}{b}{c}
\ncarc[arcangleA=-30,arcangleB=-30,linewidth=1.2pt]{->}{c}{d}
\ncarc[arcangleA=-30,arcangleB=-30,linewidth=1.2pt]{->}{d}{e}
\ncarc[arcangleA=-30,arcangleB=-30,linewidth=1.2pt]{->}{e}{f}
\ncarc[arcangleA=-30,arcangleB=-30,linewidth=1.2pt]{->}{f}{g}
\ncarc[arcangleA=-30,arcangleB=-30,linewidth=1.2pt]{->}{g}{h}
\ncarc[arcangleA=-30,arcangleB=-30,linewidth=1.2pt]{->}{h}{i}
\ncarc[arcangleA=-30,arcangleB=-30,linewidth=1.2pt]{->}{i}{Jj}

    \end{pspicture}
\end{pdfpic}
\end{center}
\caption{\label{fig:mountain} The mountain climbing lemma: two mountain climbers starting at the same height, which is a global minimum for the mountain range, can coordinate their movements so that they are at all times at equal heights, and at some point reach a common summit.}
\end{figure}

To state the version we will use, define, for any integrable function $f:S^1\to \R$ whose integral on the whole circle is 0, the set of intervals $\cI^f$ to be those intervals of the circle on which the integral of $f$ is 0.  We include singletons and the whole circle in $\cI^f$ as degenerate cases.
 We then  have the following:
\begin{theorem}\label{t.mountain}
  Let $f:S^1\to \{-1,+1\}$ be any function on the circle with finitely many discontinuities, and $\int_{S^1} f(x)dx=0.$  Then there is a continuous function $I:[0,T]\to \cI^f$ with $I(0)$ a singleton and $I(T)=S^1$.\qed
\end{theorem}
The connection between this and the version of the mountain climbing lemma depicted in Figure \ref{fig:mountain} is made by defining a function $F$ on the circle by $F(x)=\int_{t=0}^x f(x)$, where integration is done, say, in the counterclockwise direction.  Note that $F$ is continuous and in our setup, is piecewise linear with finitely many linear pieces. Taking $x_0$ to be the location of a minimum of this function, the mountain climbing lemma ensures that there are continuous paths $\alpha:[0,1]\to S^1$ and $\beta:[0,1]\to S^1$ from $x_0$ which never cross $x_0$, which eventually meet, and which have $F(\alpha(t))=F(\beta(t))$ at all times.  These functions give the endpoints of the intervals defined by the function $I(t)$.

Our analysis in this paper will make use of the following discrete version of Theorem \ref{t.mountain}:
\begin{lemma}\label{t.discreteint}
  Let the vertices $v_1,v_2,\dots,v_n$ of the $n$ cycle $C_n$ (indexed in cyclic order) be assigned labels $\ell(v_i)\in [-K,K]$, such that
  \[
  \sum_{i=1}^n\ell(v_i)=r.
  \]
  Then there is a sequence of subsets 
  \[
  \varnothing=X_0,X_1,X_2,\dots,X_T=V(C_n)
  \]
  of $V(C_n)$, such that each $X_i$, $0< i<T$ is a path, such the symmetric difference of $X_i$ and $X_{i+1}$ is single vertex for all $0\leq i<T$, and such that
  \[
  \forall 0\leq i\leq T, \quad \left|\sum_{v\in X_i} \ell(v)-\frac{|X_i|}{n}r \right|\leq 2K.
  \]
\end{lemma}
\begin{proof}
  Let $L=\sum_i |\ell(v)_i|$, and, with the aim of applying Theorem \ref{t.mountain}, recursively define intervals $I_0,I_1,\dots,I_n$ corresponding to each vertex $v_i$ of $C_n$, each of the form $[x_i,x_{i+1})$, where the length $x_{i+1}-x_i$ of $I_i$ is given by
    \[
    x_{i+1}-x_i=\frac 1 L |\ell(v)_i|.
    \]
    We then define $f$ on $S^1$ to be 1 on intervals $I_i$ whose corresponding vertex $v_i$ has a positive label $\ell(v_i)$, and $-1$ otherwise.    Theorem \ref{t.mountain} gives us the continuous interval function $I:[0,T]\to \cI^f$; let the corresponding left- and right-endpoint functions be $a(t)$ and $b(t)$.

    Now define $\cX(t)$ to be the cyclically ordered list intervals of $I_i$ intersecting $I(t)$.   We define a graph whose vertex set $Y\subseteq V(C_n)\times V(C_n)$ consisting of all the pairs $(v_i,v_j)$ for which     there exists a $t$ such that $\cX(t)=(I_i,I_{i+1},I_{i+2},\dots,I_j)$.  Define two vertices $(v_i,v_j)$ and $(v_s,v_t)$ in this graph to be adjacent if there is a $t_0$ such that for all $\delta>0$, $\cX\big((t_0-\delta,t_0+\delta)\big)$ includes both $(I_i,\dots,I_j)$ and $(I_s,\dots,I_t)$.  By continuity of $I(t)$, we have that $s\in \{i-1,i,i+1\}$ and $t\in \{j-1,j,j+1\}$.  Viewing the points on the edges and vertices of this graph as a topological space, $I(t)$ induces a continuous path in this topological space from a pair $(v_i,v_i)$ whose corresponding list of intervals is a singleton, to a pair $(v_i,v_{i-1})$ whose corresonding list includes every interval.  In particular, these two vertices lie in the same connected component of this graph and there is a (discrete) path in the graph from one to the other.  To obtain the statement of the theorem, we replace any edges of the path $(v_i,v_j)\sim (v_s,v_t)$ for which $s=i\pm 1$ and $t=j\pm 1$ with a path of length 2 of the form $(v_i,v_j)\sim (v_i,v_t)\sim (v_s,v_t)$.

    Since the integral of $\int_{I(t)}f(x)dx=0$ for any $t$, and the sum of the lengths of the intervals in $\cX(t)$ differs from this integral by at most $2K$ for any $t$, this path gives the sequence of subsets claimed by the theorem.
\end{proof}

The next lemma is a consequence of the Local Lemma, and will be used to control the congestion near the endpoints of our canonical paths.
\begin{lemma}\label{local}For $\Delta$ sufficiently large (e.g., at least $e^{e^4}$), we have that for any bipartite graph on the bipartition $A\,\dot\cup \,B$ of maximum degree at most $\Delta$, where vertices in $A$ have degree at least $100 \Delta/\log\Delta$, there is a mapping $\theta:A\to B$ such $(a,\theta(a))$ is an edge for all $a\in A$, and every preimage $\theta^{-1}(b)$ has size at most $\tfrac 1 2 \log \Delta$.%, and vertices in $B$ have degree at most $\beta$.
\end{lemma}
\begin{proof}
Consider assigning to each vertex $a\in A$ a random vertex in its neighborhood, uniformly at random, and making all $|A|$ of these choices independently.

We associate to each vertex $b\in B$ a bad event $\cE_{b,D}$, which is the event that more than $D$ vertices $a\in A$ (with $a\sim b$) get assigned the vertex $b$.  For fixed $D$, we obtain a valid dependency graph in the sense of the Lov\'asz Local Lemma for this collection of bad events by letting $\cE_{b,D}\sim \cE_{b',D}$ if there is an $a$ such that $b\sim a\sim b'$; the degree of this dependency graph is thus at most $\Delta^2$.

We bound the probability of each $\cE_{b,D}$ by comparison with a Poisson distribution.
For 
\[
\lambda=-\ln(1-p)%=\sum_{k\geq 1}\frac{p^k}{k}\leq \frac{p}{1-p}
\]
we have $e^{-\lambda}=(1-p)$ and so a Bernoulli with success probability $p$ is dominated by a Poisson distribution of rate $\lambda$ ($\Pr(Po=0)=1-p=\Pr(Be=0)$).  Thus a random variable distributed as Binomial$(n,p)$ is dominated in distribution by a random variable distributed as Poisson with mean $\lambda=-n\ln(1-p)$. Finally, recall that the Poisson random variable $\nu$ with mean $\lambda$ satisfies the upper tail bound
\begin{equation}\label{upperpos}
\Pr(\nu\geq K)\leq e^{-\lambda}\left(\frac{e\lambda}{K}\right)^K.
\end{equation}
In our case $n=\Delta$ and letting $\alpha=100/\log\Delta$ and $p=\frac{1}{\alpha \Delta}$ and using the bounds
\[
\Delta p<\lambda<\Delta p/(1-p),
\]
we have
\begin{multline*}
\Pr(\cE_{b,D})\leq e^{-{1/\alpha}}\left(\frac{1}{1-\frac{1}{\alpha\Delta}}\right)^{D}\left(\frac{e}{\alpha D}\right)^D\\
=e^{-{1/\alpha}}\left(1+\frac{1}{\alpha \Delta-1}\right)^{D}\left(\frac{e}{\alpha D}\right)^D\leq e^{D/(\alpha \Delta-1)}\frac{e^D}{\alpha^D}\frac{1}{D^D}
\\\leq \frac{e^{2D}}{\alpha^D}\frac{1}{D^D}\leq \left(\frac{2}{D}\right)^D
\end{multline*}
for $\alpha\Delta\geq 2$ and $D\geq 2e^2/\alpha$.  So for $D=\log \Delta,$ $\alpha= \frac {100} {\log \Delta}$,
\[
\Pr(\cE_{b,D})\leq \frac{1}{\Delta^{(\log \log \Delta)-1}},
\]
and so for, say, $\Delta\geq e^{e^4}$, we have $ed\Pr(cE_{b,D})<1$ in the application of the Local Lemma (recall the dependency graph has degrees less than $d=\Delta^2)$ and so the Local Lemma applies to show the existence of the claimed subgraph.  Note that $\Delta\geq 2^{10}$ follows already from the assumption that $\Delta\geq \alpha=\frac {100} {\log \Delta}$.\\
\end{proof}

The final lemma in this section simply captures the balance we can ensure when partioning arbitrary connected graphs of bounded degree into connected pieces.
\begin{lemma}\label{l.treebreak}
If $T$ is a tree on $n\geq 2$ vertices and of maximum degree at most $\Delta$ with root $v_0$ of degree $\leq \Delta-1$, then for any $B<n$, we can partition $T$ into subtrees $T''$ and $T'$ so that $\frac{B-1}{\Delta-1}< |V(T'')|\leq B$ and $v_0\in T'$.
\end{lemma}
\begin{proof}
Consider the trees $T_1,T_2,\dots,T_d$ ($d\leq \Delta-1$) that are the connected components of $T\setminus v_0$, and let $v_i$ be the neighbor of $v_0$ in $T_i$.  Since $\frac{n-1}{\Delta-1}>\frac{B-1}{\Delta-1}$, there is some $i$ such that $|V(T_i)|>\frac{B-1}{\Delta-1}$.  If  $|V(T_i)|\leq B$ as well, we're done; otherwise, we apply induction to the tree $T_i$ rooted at $v_i$.
\end{proof}

\subsection{Repartionable graphs}
\label{sec:repartitionable}

In this section we revisit Definition \ref{def:partitionable}. We begin by proving Proposition \ref{p.connectedgraphs}, showing that graphs of maximum degree $\Delta$ are $(B,\frac{1}{\Delta-1}\frac{B-1}{B})$-partitionable. 
\begin{proof}[Proof of Proposition \ref{p.connectedgraphs}]
Applying Lemma \ref{l.treebreak} recursively shows that we can partition $G$ into pieces of size greater than $\frac{B-1}{\Delta-1}$ and at most $B$.  Moreover this can be done in polynomial time since at each application, there are only linearly many choices for the division claimed to exist by the Lemma.  This proves the first part of the definition holds as claimed.

We will show the second part of the definition is satisfied for any 
\begin{equation}\label{e.K}
K=\frac{2 B\Delta}{\ep}+1.
\end{equation}
Let $G$ be any connected graph and $c$ any $B$-bounded contiguous-$\kappa$-coloring of $G$ with \[
\kappa\geq \frac{(1+\ep)n(\Delta-1)}{B-1}.
\]
Define the graph $H$ whose vertex-set is the color classes of $c$, with two color classes being adjacent in $H$ if they are joined by any edge in $G$.  Note that the maximum degree of $H$ is at most $B\cdot \Delta$.  Letting $k=\lceil\tfrac 2 {\ep} \rceil$, we see from Lemma \ref{l.treebreak} that we can partition $H$ into connected pieces $H_1,\dots,H_s$ each of whose size satisfies $k<|V(H_i)|\leq k(B\Delta-1)+1$, except for one piece that may have smaller size.  In particular, by merging the small piece with another, we may assume that each piece satisfies 
\[
k<|V(H_i)|\leq kB\Delta+1.
\]
Each $H_i$ corresponds to a subgraph $G(H_i)$ of $G$ whose vertex set is the union of some color classes of $c$.  Since the average size of a color class of $c$ is less than $\frac{1}{1+\ep}\frac{B-1}{\Delta-1}$, there is a $\delta>0$ (depending on $B$, $\Delta$, $\ep$) and $\delta|V(G)|$ indices $i_1,i_2,\dots$ such that $G(H_{i_j})$'s is a union of color classes of $c$ which have average size less than $\frac{1}{1+\ep/2}\frac{B-1}{\Delta-1}$, for all $j$.

By Lemma \ref{l.treebreak}, this time applied to a $G(H_{i_j})$, there is a partition of $G(H_{i_j})$ into pieces of size at most $B$ and size at least $\frac{B-1}{\Delta-1}$, except for one piece that may have smaller size; if the partition of $G(H_{i_j})$ has $\leq k$ classes then we're done, as this partition can be used to give an alternative coloring $c'$ which agrees with $c$ except on $G(H_{i_j})$, which satisfies $|G(H_{i_j})|\leq K$ by \eqref{e.K}.  But the partition cannot have more than $k$ classes, since then then the average size of a piece in our partition would be at least
\[
\frac{k}{k+1}\frac{B-1}{\Delta-1}\geq \frac{1}{1+\ep/2}\frac{B-1}{\Delta-1}.
\]
This for each of the $\delta_0|V(G)$ indices $i_j$, we obtain a distinct contiguous coloring using fewer than $\kappa$ colors, and differing from $c$ on a connected subgraph of most $K$ vertices.
\end{proof}
Next we prove Proposition \ref{p.grids}, asserting that grids are $(B,1)$-partitionable for all $B$.
\begin{proof}[Proof of Proposition \ref{p.grids}]
The first part of the definition follows from the fact that grid graphs have Hamilton paths, and paths can be partitioned as claimed.

Next, let $G$ be a $k\times \ell$ grid on $n=k\ell$ vertices and assume $k\leq \ell.$  
%Given $B$ and $\ep>0$, write $\rho=\frac{\ep}{2B}$.  
We define parameters in two cases:\\
\textbf{Case 1:} If $k\leq 400B^2/\ep^2$, %$k\leq 2\lceil\tfrac{B}{\ep}\rceil/\rho$, 
we set $L=k$, and $M=\lceil\tfrac{20B}{\ep}\rceil$.\\
\textbf{Case 2:} If $k,\ell>400B^2/\ep^2$, %$k,\ell> 2\lceil\tfrac{B}{\ep}\rceil/\rho$, 
we set $L=M=\lceil\tfrac{20B}{\ep}\rceil$.\\

By the choice of $L$, there is a $k'\times \ell'$ subgrid $G'$ of $G$ with $k'$ and $\ell'$ multiples of $L$ and $M$, respectively, and for which $n':=k'\ell'\geq n-40B/\ep$.  In particular, for $n$ sufficiently large, given any $B$-bounded contiguous-$\kappa$-coloring $c$ of $G$ for $\kappa\geq (1+\ep)n/B$, at least $(1+\ep/2)n/B$ of its color classes intersect $G'$.  We can then partition $G'$ into $L\times M$ grid subgraphs $H_i$.  In Case 1 the area of each such grid subgraph is $\geq 20kB/\ep$ and for each $H_i$, the set $B(H_i)$ of boundary vertices---i.e., vertices that have neighbors in other $H_j$'s---has size at most $2k$ for all $i$.  In Case 2, the area of each $H_i$ is $\geq 400B^2/\ep^2$ and each boundary $B(H_i)$ has size at most $80B/\ep+4$.  In both cases, we have
\begin{equation}\label{BHi}
\frac{|B(H_i)|}{|V(H_i)|}\leq \frac{\ep}{4B}.
\end{equation}
Thus, fixing a $B$-bounded contiguous-$\kappa$-coloring, at most $\ep n/4B$ color classes can intersect a boundary vertex of an $H_i$.  Subtracting from the $\geq (1+\ep/2)n/B$ color classes that intersect $G'$, we are left with at least 
\[
(1+\ep/4)n/B\geq (1+\ep/4)n'/B
\]
color classes that are subsets of $V(G')$ and intersect exactly one $H_i$.  Since no $H_i$ can contain more than $L\cdot M$ color classes, we have for
\[
\delta=\frac{\ep}{8BLM}
\]
that there are $\delta n$ indices $i$ such that there are at least $(1+\ep/8)LM/B>LM/B$ color classes which are subsets of $H_i$.  Each of these can be used to provide a contiguous-$\kappa'$-coloring for $\kappa'<\kappa$ satisfying the definition.  Indeed, as each such $H_i$ is Hamiltonian, each can be partitioned into contiguous pieces are arbitrary sizes summing to $|V(H_i)|$; in particular, each can be partitioned into color classes of size exactly $B$, which in each case requires just $LM/B$ color classes per $H_i$.
\end{proof}

The first part of Definition \ref{def:partitionable} ensures that we can efficiently find an element of $\Omega_{q,B,G}$ as a starting point for our algorithm.  Theorems \ref{t.part} and \ref{t.color} will then follow from bounds on the mixing times of the Markov chain, which we achieve using the method of canonical paths.
\subsection{Outline of the proof strategy}
We bound the mixing time of $\cM$ using the method of canonical paths \cite{diaconis1991geometric,sinclair1992improved,jerrum1989approximating}, which requires us to define a path between any pair of states so that no edge is used by too many paths. 

The basic idea of our approach to defining canonical paths is simple. Given the graph $G$, we use the bandwidth ordering $\sigma$ and divide the vertex set $G$ into intervals of length $2B\band(G)$ in the ordering $\sigma$; this ensures that any piece of a $B$-bounded contiguous $q$-partition intersects at most two of these intervals.  

Next, given two valid partitions $\omega$ and $\omega'$ of $G$, the basic idea is that we would like to, one interval at a time, replace the classes of $\omega$ in one interval with the classes of $\omega'$.   First (in Phase 0 below), we must reduce the number of nonempty partition classes so that we have enough room, e.g., to break classes into singletons.

But when carrying out these copy operations, the crucial thing we need to maintain is that all intermediate states use few enough partition classes that they are valid members of the state space, but also sufficiently many partition classes that they are not a high-congestion node (note that we should expect there to be few partitions which use few classes; thus it is not surprising that canonical paths cannot make heavy use of such partitions as intermediate states).  This is where the mountain climbing problem enters.  Lemma \ref{t.discreteint} ensures that there is some order in which we can copy (and uncopy!) the partition $\omega'$ into the partition $\omega$ so that we eventually have entirely replaced $\omega$ with $\omega'$, but at all intermediate states have a ``just right'' number of nonempty partition classes.
\subsection{Defining Valid Canonical Paths}
For $G$ an $(\alpha,B)$-repartionable graph, and any $\ep>0$, let $K$ be as given in Definition \ref{def:partitionable}.  We now define canonical paths between any states $\omega,\omega'$ of $\Omega=\Omega_{q,B,G}^K$, for $q>(1+\ep)n/(\alpha B)$.  Define $\kappa(\omega)$ to be the number of non-null color classes $\omega$.  We let $t=2B\band(G)$.

\subsubsection*{Phase 0}
From the definition of $(B,\alpha)$-repartionable and our choice of $q>(1+\ep)n/(\alpha B)$, we can define paths
\begin{equation}\label{phase0om}
\omega=\omega_0,\omega_1,\dots,\omega_{s_0}
\end{equation}
and
\begin{equation}\label{phase0om'}
\omega'=\omega'_0,\omega'_1,\dots,\omega'_{t_0}
\end{equation}
in $\Omega$ such that
$\omega_{s_0}$ and $\omega_{t_0}$ both have fewer than $q-10t$ nonempty color classes, and where $s_0,t_0\leq 10Kt$. For example, applying the definition of $(B,\alpha)$-partitionable to $\omega$ gives us a sequence $\omega=\omega_0,\dots,\omega_k$, for $k<K$, where $\omega_k\in \Omega_{q-1,B,G}$ and for all $0<i<k$ we have $\omega_i\in \Omega^K_{q-1,B,G}$. Applying this repeatedly, we eventually obtain a sequence ending in the state $\omega_{s_0}\in \Omega_{q-10t,B,G}$, with all intermediate states in $\Omega^K_{q-1,B,G}$.

\bigskip
To reduce congestion, we choose the sequences \eqref{phase0om},\eqref{phase0om'} carefully, as follows.  For each $\kappa$ with $q-10t\leq \kappa\leq q$, we define $\Omega_\kappa\subseteq \Omega$ to be those colorings with exactly $\kappa$ nonempty color classes.  We define a bipartite graph on $A\,\dot\cup\,B$ for $A=\Omega_{\kappa},B=\Omega_{\kappa-1}$ by letting $\omega$ be adjacent to $\omega'$ in this graph for $\omega\in \Omega_{\kappa}, \omega'\in \Omega_{\kappa-1}$ if $\omega'$ differs from $\omega$ only on a connected subgraph of at most $K$ vertices.  From Definition \ref{def:partitionable}, there is $\delta>0$ such that the degree of a vertex $\omega\in A$ is least $\delta n$.  On other hand, very crudely, the number of subsets of $V(G)$ of size $K$ inducing connected subgraphs of $G$ is at most $n\Delta^K$; this gives an upper bound on the maximum degree of the bipartite graph.  Thus by Lemma \ref{local} and the upper bound on $\Delta$ of order $(\delta\log n)^{1/K}$
\[
\delta n\geq \frac{100 n\Delta^K}{\log(100 n \Delta^K)}
\]
for sufficiently large $n$,
there is a mapping $\theta_\kappa:\Omega_{\kappa}\to \Omega_{\kappa-1}$ such that $\theta_\kappa(\omega)$ differs from $\omega$ only on a connected subgraph of at most $K$ vertices, and such that for all $\omega'\in \Omega_{\kappa-1}$, the preimage $\theta_\kappa^{-1}(\omega')$ has size at most 
\begin{equation}\label{thetabound}
\frac 1 2\log n +\frac12 K\log \Delta\leq \log n.
\end{equation}
We fix choices of the maps $\theta_\kappa$ for each $\kappa$, and use these maps when choosing sequences in \eqref{phase0om} and \eqref{phase0om'}.

\noindent For notational convenience, we write $\bar\omega=\omega_{s_0}$ and $\bar \omega'=\omega'_{t_0}$.

\subsubsection*{Setting up the next step}
It remains to define a path from $\bar \omega$ to $\bar \omega'$.

We let $\Phi(\omega)$ be the set of nonempty color classes of $\omega$ and let $\Phi_i(\omega)$ denote the set of color classes $C\in \Phi(\omega)$ for which
\[
\min_{v\in C}\left\lfloor\frac{\sigma(v)}{t}\right\rfloor=i.
\]
Observe that for $i=0,\dots,m-1$ ($m=\lfloor n/t\rfloor+1$) %the $\Phi^{(t,h)}_i(\omega)$ form a partition of $\Phi^{(t,h)}(\omega)$, and 
the $\Phi_i(\omega)$ form a partition of $\Phi(\omega)$.  We let $V(\Phi_i(\omega))$ denote the underlying vertex set of $\Phi_i(\omega)$; that is, the union of its elements.  Observe that by our choice of $t$, we have
\begin{equation}\label{tvcover}
(\forall \omega\in \Omega)\quad V(\Phi_i(\omega))\subseteq V_i\cup V_{i+1}.
\end{equation}
\begin{equation}\label{tcover}
(\forall \omega,\omega'\in \Omega)\quad V(\Phi_i(\omega'))\subseteq V(\Phi_{i-1}(\omega))\cup V(\Phi_{i}(\omega))\cup V(\Phi_{i+1}(\omega)),
\end{equation}
We write $\phi_i(\omega)$ for the size of $\Phi_i(\omega)$; i.e., the number of color classes it contains, and similarly write $\phi(\omega)$ for the size of $\Phi(\omega)$.

Now we label the vertices $v_0,v_1,\dots,v_m=v_0$ of the cycle $C_m$ (indexed in cyclic order) with the labels 
\[
\ell(v_i)=\phi_i(\bar \omega')-\phi_i(\bar \omega),
\]
which satisfy
\[
|\ell(v_i)|\leq t
\]
and let 
\begin{equation}\label{Xpath}
X_0,X_1,\dots,X_T
\end{equation}
be the sequence of subsets of the vertices of the cycle guaranteed to exist by Lemma \ref{t.discreteint}, with $T\leq m^2/2.$

Recall that each $X_i$ in $X_1,\dots,X_{T-1}$ induces a subgraph of $C_m$ that is a path, whose vertex set we call $I_i$.  Its complement is also a path, and we denote the interior of this complementary path by $O_i$.  Finally, we let $B_i=C_m\setminus (I_i\cup O_i$); so, for all $i=1,\dots,T-1$, $|B_i|=2$, and in this case we write $b_k^+$ and $b_k^-$ for the vertices in $B_k$ which are adjacent to the greatest and least vertices of $I_k$, respectively, in cyclic order.

Now with Phase 1 beginning from the state $\bar \omega_0=\bar \omega$, we will define Phases $1,\dots,T$ so each Phase $k$ ends at a state $\bar \omega_k$ such that:
\begin{enumerate}[(a)]
\item \label{agreesIk}
 For each $v_i\in I_k$, every color class $C'\in \Phi_i(\bar \omega')$ is a color class of $\bar\omega_k$.
\item \label{agreesOk}
 For each $v_i\in O_k$, every color class $C\in \Phi_i(\bar \omega)$ is a color class of $\bar \omega_k$.
\item \label{singleBk} For $v_i=b_k^+$, every vertex in $\big(V(\Phi_{i}(\bar\omega))\cup V(\Phi_{i-1}(\bar \omega))\big)\setminus V(\Phi_{i-1}(\bar\omega'))$ is a singleton color class in $\bar \omega_k$, while for $v_i=b_k^-$, every vertex in $\big(V(\Phi_{i}(\bar\omega))\cup V(\Phi_{i+1}(\bar \omega))\big)\setminus V(\Phi_{i+1}(\bar\omega'))$ is a singleton color class in $\bar \omega_k$.
\end{enumerate}

In particular , these conditions imply that
\begin{equation}\label{kappalemma}
\phi(\bar \omega_k)=\sum_{v_i\in O_k} \phi_i(\bar \omega)+\sum_{v_i\in I_k} \phi_i(\bar \omega')+E_k
=\sum_{v_i\in I_k\cup O_k} \phi_i(\bar \omega)+\sum_{v_i\in I_k} \ell(v_i)+E_k,
\end{equation}
for 
\[
|E_k|\leq 2t.
\]
In particular, we have with $r=\phi(\bar\omega')-\phi(\bar\omega)$ in the statement of Lemma \ref{t.discreteint}
\begin{equation}\label{q-6t}
\phi(\bar\omega_k)\leq 
\phi(\bar\omega)+\frac{|I_k|}{m}(\phi(\bar\omega')-\phi(\bar\omega))+4t
\leq\max\big(\phi(\bar\omega),\phi(\bar\omega')\big)+4t\leq q-6t.
\end{equation}

\subsubsection*{Phase $k$}
Inductively, having completed Phase $k-1$ (or, for $k=1$, beginning from $\bar\omega_0=\bar\omega$, we proceed as follows.  

Recall the $X_i$ from \eqref{Xpath} which we use to define our phases, and the associated sets $I_i,O_i,B_i$.  $X_k$ differs from $X_{k-1}$ in a single vertex, in the sense that there is $v_{i_k}\in B_{k-1}$ and an $i_k'=i_k\pm 1$, so that either $i_k'\in O_{k-1}$ and
\begin{equation}\label{addtoI}
B_k=(B_{k}\setminus\{v_{i_k}\})\cup \{v_{i_k'}\}\quad I_k=I_{k-1}\cup \{v_{i_k}\}\quad O_k=O_{k-1}\setminus \{v_{i_k'}\} 
\end{equation}
or else $i_k'\in I_{k-1}$ and
\begin{equation}\label{addtoO}
B_k=(B_{k}\setminus\{v_{i_k}\})\cup \{v_{i_k'}\}\quad I_k= I_{k-1}\setminus \{v_{i_k}\}\quad O_k=O_{k-1}\cup \{v_{i_k'}\}.
\end{equation}
Consider the first case.  Here, we proceed as follows.
\begin{enumerate}[(i)]
\item \label{step.shatter} By \ref{agreesOk}, $\bar \omega_{k-1}$ agrees with $\bar \omega$ on $\Phi_{i_k'}(\bar \omega)$.  One by one, for each color class in $\Phi_{i_k'}(\bar \omega)$, color each of the vertices of the color class with a distinct unused color (this can be done recursively, beginning with the leaves of a spanning tree of the color class, as to not disconnect it) so that all vertices in these color classes become singleton color classes.  This increases the number of colors used by less than $2t$, so we will never use more than $q$ colors.  Call the result $\bar\omega_{k-1}'$.
\item By %\ref{agreesOk} and 
\ref{singleBk} and \eqref{tcover}, after completing Step \ref{step.shatter}, every vertex in $V(\Phi_{i_k}(\bar\omega'))$ is a singleton color class in $\bar\omega_{k-1}'$.  Now, one by one, we merge these singletons (by assigning them identical colors) so that the resulting coloring agrees with $\bar\omega'$ on all color classes in $\Phi_{i_k}(\bar\omega)$.    This step only decreases the number of colors used, so we will not exceed $q$ colors.
\end{enumerate}
The resulting coloring is $\bar\omega_k$, and by construction satisfies \ref{agreesIk}, \ref{agreesOk}, and \ref{singleBk}.  In the second case, the analogous steps are:

\begin{enumerate}[(i)]
\item \label{step2.shatter} By \ref{agreesIk}, $\bar \omega_{k-1}$ agrees with $\bar \omega$ on $\Phi_{i_k'}(\bar \omega')$.  One by one, for each color class in $\Phi_{i_k'}(\bar \omega)$, color each of the vertices of the color class with a distinct unused color, increasing the number of colors used by less than $2t.$  The result $\bar\omega_{k-1}'$.
\item By %\ref{agreesOk} and 
\ref{singleBk}, after completing Step \ref{step.shatter}, every vertex in $V(\Phi_{i_k}(\bar\omega'))$ is a singleton color class in $\bar\omega_{k-1}'$.  Now, one by one, we merge these singletons (by assigning them identical colors) so that the resulting coloring agrees with $\bar\omega'$ on all color classes in $\Phi_{i_k}(\bar\omega)$.
\end{enumerate}

The steps between phases never increase the number of non-null color classes by more than $2t$, and at the boundaries $\bar\omega k$ we have $\phi(\bar\omega_k)\leq q-6t$ as given in \eqref{q-6t}, which followed from our use of Lemma \ref{t.discreteint}.

\subsection{Bounding the congestion}
In the previous section we defined paths $\Gamma_{\omega,\omega'}$ between any states $\omega,\omega'\in \Omega$, with the length $|\Gamma_{\omega,\omega'}|$ satisfying
\[
|\Gamma_{\omega,\omega'}|\leq 2tm^2+20t\leq 2n^2,
\]
as Phase 0 consumes at most $20t$ steps, and each of the $T=\binom{m}{2}$ Phases 1,\dots,$T$ consumes at most $4t$ steps.

Now, letting $E$ denote the graph of transitions of the Markov Chain $\cM$, we aim to bound the congestion
\begin{equation}
B=\max_{e\in E}\,\frac{2|E|}{|\Omega|^2}\sum_{\substack{\omega,\omega'\in \Omega\\\Gamma_{\omega,\omega'\ni e}}}|\Gamma_{\omega,\omega'}|\leq \frac{4n^2|E|}{|\Omega|^2}\cdot \max_{e\in E}\#\set{\{\omega,\omega'\}\mid \Gamma_{\omega,\omega'}\ni e}.
\end{equation}
Since $|E|/|\Omega|$ is polynomial in $n$, we are particularly interested in bounding the ratio
\begin{equation}
    \frac{\max_{e\in E}\#\set{\{\omega,\omega'\}\mid \Gamma_{\omega,\omega'}\ni e}}
    {|\Omega|},
    \end{equation}
which we bound by the vertex-congestion
\begin{equation}
    \frac{\max_{\omega_1\in \Omega}\#\set{\{\omega,\omega'\}\mid \Gamma_{\omega,\omega'}\text{ visits }\omega_1}}
    {|\Omega|}.
    \end{equation}
To this end, fix an $\omega_1\in \Omega$.  By the definition of the canonical paths, if $\omega_1$ is visited by $\Gamma_{\omega,\omega'}$ then
\begin{enumerate}[(A)]
    \item There are $\bar\omega$ and $\bar\omega'$ which can be reached from $\omega$ and $\omega'$, respectively, from a sequence of at most $10t$ applications of the maps $\theta_\kappa$ defined in Phase 0, such that $\phi(\bar\omega),\phi(\bar\omega')\leq q-10t$, and such that:
    \item \label{part.IO} There are sets $I,O\subseteq C_m$ inducing connected subgraphs of $C_m$ with $|I|+|O|=m-3$, such $I$ and $O$ are separated by at least one vertex on each side, and such that $\omega_1$ agrees with $\bar\omega'$ on all color classes in $\Phi_i(\bar\omega')$ for all $v_i\in I$, and with $\bar\omega$ on all color classes in $\Phi_i(\bar\omega)$ for $v_i\in O$.\\(Note that we have $|I|+|O|=m-3$ instead of $|I|+|O|=m-2$ because $\omega_1$ can come from the middle of a phase rather than an endpoint.)
    \item \label{part.fromlem} From our application of Lemma \ref{t.discreteint} in our construction of the canonical paths, these sets satisfy
    \begin{equation*}\label{applemmaI}
\sum_{v_i\in I}(\phi_i(\bar\omega')-\phi_i(\bar\omega))=\frac{\abs{I}}{m}\big(\phi(\bar\omega')-\phi(\bar\omega)\big)\pm E\quad\text{for}\quad |E|\leq 4t,
\end{equation*}
and thus, by subtracting both sides from $\phi(\bar\omega')-\phi(\bar\omega)$, also satisfy
\begin{equation*}\label{applemmaO}
\sum_{v_i\in O}(\phi_i(\bar\omega')-\phi_i(\bar\omega))=\frac{\abs{O}}{m}\big(\phi(\bar\omega')-\phi(\bar\omega)\big)\pm F\quad\text{for}\quad |F|\leq 8t.
\end{equation*}
\item \label{sizeconstraints} We thus have that
\begin{equation}\label{applemmaI1}
\sum_{v_i\in I}\phi_i(\omega_1)-\phi_i(\bar\omega)=\frac{\abs{I}}{m}\big(\phi(\omega_1)-\phi(\bar\omega)\big)\pm E,
\end{equation}
\begin{equation}\label{applemmaO1}
\sum_{v_i\in O}\phi_i(\bar\omega')-\phi_i(\omega_1)=\frac{\abs{O}}{m}\big(\phi(\bar\omega')-\phi(\omega_1)\big)\pm F.
\end{equation}
\end{enumerate}
\bigskip

Our immediate task is to bound the number of choices for $\bar \omega$ and $\bar \omega'$ given $\omega_1$.  There are less than $m^2$ choices for the sets $I$ and $O$ as in \ref{part.IO}.  Then, given $I$ and $O$, all color classes of $\bar\omega'$ in sets $\Phi_i(\bar\omega')$ for $v_i\in I$ appear as color classes of $\omega_1$; all that is unknown is the color classes in $\Phi_i(\bar\omega')$ for $v_i\in C_m\setminus I.$  Similarly, all color classes of $\bar\omega$ in sets $\Phi_i(\bar\omega)$ for $v_i\in O$ appear as color classes of $\omega_1$; all that is unknown is the color classes in $\Phi_i(\bar\omega)$ for $v_i\in C_m\setminus O$.

Thus we will bound the number of choices for $\bar\omega,\bar\omega'$ by bounding the choices for the sets
\begin{equation}\label{thosesets}
    \Phi_i(\bar\omega')\text{ for }v_i\in C_m\setminus I,\quad \Phi_i(\bar\omega)\text{ for }v_i\in C_m\setminus O.
\end{equation}
We will do this by first bounding the choices for the slightly smaller collection of sets
\begin{equation}\label{thesesets}
    \Phi_i(\bar\omega')\text{ for }v_i\in O,\quad \Phi_i(\bar\omega)\text{ for }v_i\in I.
\end{equation}
By \eqref{tcover}, these sets satisfy the property that no $C\in \Phi_i(\bar\omega)$ for $v_i\in I$ can intersect any $C'\in \Phi_j(\bar\omega')$ for $v_j\in O$, and by \eqref{tvcover} that the only vertices not assigned a color class in either collection of sets belong to $V_{b^+}\cup V_{b^++1}\cup V_{b^-}\cup V_{b^-+1}$ where $\{b^+,b^-\}=C_m\setminus I\setminus O$.   Thus, after fixing a choice for the color classes in the sets in \eqref{thesesets}, the result can be extended to a partition of the whole graph $G$ by letting the unassigned vertices of $V_{b^+}\cup V_{b^++1}\cup V_{b^-}\cup V_{b^-+1}$ be singleton classes.  Doing so actually gives a valid $q$-partition of of $G$, since this requires at most $4t$ singleton classes, and by \ref{sizeconstraints}, the total number of color classes in the sets in \eqref{thesesets} must satisfy
\begin{multline}
    \sum_{v_i\in I} \phi_i(\bar \omega)+\sum_{v_i\in O} \phi_i(\bar\omega')
    \leq
    \frac{|I|}{m}\phi(\bar\omega)+\frac{|O|}{m}\phi(\bar\omega')
    \leq \max(\phi(\bar\omega),\phi(\bar\omega'))\\\leq q-10t.
\end{multline}
This implies that the number of choices for the sets in \eqref{thesesets} is at most $|\Omega|$.

$\Delta^n$ gives a crude upper bound on the number of forests in a graph on $n$ vertices with maximum degree $\leq \Delta$.  In turn this lets us bound the number of choices for the four sets
\begin{equation}\label{foursets}
    \Phi_i(\bar\omega')\text{ for }v_i\in C_m\setminus I\setminus O,\quad \Phi_i(\bar\omega)\text{ for }v_i\in C_m\setminus O\setminus I.
\end{equation}
by $\Delta^{2t}$ each, for a total of at most $\Delta^{8t}$ choices.  Thus in total we've bounded the number of choices for $\bar\omega$ and $\bar\omega'$ by 
\[
m^2\Delta^{8t}|\Omega|.
\]
Finally, accounting for Phase 0, the bound in \eqref{thetabound} implies that there at most
\[
(\log n)^{20t}m^2\Delta^{8t}|\Omega|
\]
choices for the states $\omega,\omega'$.

\section{Torpid mixing for $q=2$}
\label{s.torpid}
In this section we give an example of a graph $G$ which is a subgraph of the grid with bounded bandwidth for which the Markov chain $\cM$ is ergodic, but for which the mixing time of $\cM$ is exponential in the number of vertices of $G$.

It is simpler to give an example of torpid mixing for the chain operating on colors rather than partitions, so we do this first.  (In particular, for this example, we will give a bottleneck which separates the set of colorings into two classes, which are equivalent under interchange of colors; thus this is not a bottleneck in the partition chain). For this we let $G=G_{3\times n}$ be a $3\times n$ grid graph, which has bandwidth 3.

\begin{figure}
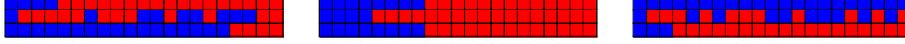

    \centering
\begin{minipage}{.33\linewidth}
\newdimen\omsq  \omsq=5pt
\newdimen\omrule    \omrule=.5pt
\newdimen\omint

\newif\ifvth    \newif\ifhth    \newif\ifomblank
\OMINO{
+-+-+-+-+-+-+-+-+-+-+-+-+-+-+-+-+-+-+-+-+-+\\
|b|b|b|b|r|r|r|r|r|r|r|r|r|r|r|r|r|r|r|r|r|\\
+-+-+-+-+-+-+-+-+-+-+-+-+-+-+-+-+-+-+-+-+-+\\
|b|r|r|r|r|r|b|r|r|r|b|b|r|b|b|r|b|b|b|r|r|\\
+-+-+-+-+-+-+-+-+-+-+-+-+-+-+-+-+-+-+-+-+-+\\
|b|b|b|b|b|b|b|b|b|b|b|b|b|b|b|b|b|r|r|r|r|\\
+-+-+-+-+-+-+-+-+-+-+-+-+-+-+-+-+-+-+-+-+-+\\
}
\end{minipage}%
\begin{minipage}{.33\linewidth}
\newdimen\omsq  \omsq=5pt
\newdimen\omrule    \omrule=.5pt
\newdimen\omint

\newif\ifvth    \newif\ifhth    \newif\ifomblank
\OMINO{
+-+-+-+-+-+-+-+-+-+-+-+-+-+-+-+-+-+-+-+-+-+\\
|b|b|b|b|b|b|b|b|r|r|r|r|r|r|r|r|r|r|r|r|r|\\
+-+-+-+-+-+-+-+-+-+-+-+-+-+-+-+-+-+-+-+-+-+\\
|b|b|b|b|r|r|r|r|r|r|r|r|r|r|r|r|r|r|r|r|r|\\
+-+-+-+-+-+-+-+-+-+-+-+-+-+-+-+-+-+-+-+-+-+\\
|b|b|b|b|b|b|b|b|r|r|r|r|r|r|r|r|r|r|r|r|r|\\
+-+-+-+-+-+-+-+-+-+-+-+-+-+-+-+-+-+-+-+-+-+\\
}
\end{minipage}%
\begin{minipage}{.33\linewidth}
\newdimen\omsq  \omsq=5pt
\newdimen\omrule    \omrule=.5pt
\newdimen\omint

\newif\ifvth    \newif\ifhth    \newif\ifomblank
\OMINO{
+-+-+-+-+-+-+-+-+-+-+-+-+-+-+-+-+-+-+-+-+-+\\
|b|b|b|b|b|b|b|b|b|b|b|b|b|b|b|b|b|b|b|b|b|\\
+-+-+-+-+-+-+-+-+-+-+-+-+-+-+-+-+-+-+-+-+-+\\
|b|r|r|r|b|r|b|r|r|r|b|b|r|b|b|b|r|b|r|b|r|\\
+-+-+-+-+-+-+-+-+-+-+-+-+-+-+-+-+-+-+-+-+-+\\
|b|b|b|r|r|r|r|r|r|r|r|r|r|r|r|r|r|r|r|r|r|\\
+-+-+-+-+-+-+-+-+-+-+-+-+-+-+-+-+-+-+-+-+-+\\
}
\end{minipage}
\caption{\label{fig:3xl} Partitions of $G=G_{3\times \ell}$ (for $\ell=21)$ into two connected color classes.  There are many more colorings which agree on the boundary vertices of $G$ with the first and third shown coloring than with the middle shown coloring, which forms a bottleneck in the chain.  In our discussion of this example, blue is color 1, and red is color 2.}
\end{figure}

When drawn in the plane in the natural way, the outer face of the $k\times \ell$ grid graph $G_{k\times \ell}$ is a $2k+2\ell-4$ cycle which we denote by $B_{k\times \ell}$.  The connectivity of the color classes implies that the intersection of each color class $\omega^{-1}(i)$ with $B_{k\times \ell}$ is a path, which we denote by $P_i=P_i(\omega)$.  In our examples there will be only polynomially many colorings for which one of these paths is empty, thus we focus on the colorings where both paths are nonempty.

We fix a cyclic orientation of the cycle $B_{k\times \ell}$ (when referring to Figure \ref{fig:3xl}, we will use counterclockwise) so that we can refer to the \emph{first} and \emph{last} vertex of each path $P_i$ in a consistent way.

Focusing now on the case $k=3$, let define $S$ to be those states $\omega\in \Omega_{3\times \ell}$ for which coordinate of the first vertex of $P_1(\omega)$ is less than the second coordinate of the last vertex of $P_1(\omega)$.  A representative element of $S$ is illustrated at left in Figure \ref{fig:3xl} (note that we show the subgraph of the dual of the grid corresponding to $G$, and color its faces).

By symmetry we have that $\pi(S)=\frac 1 2(1-\rho)$ where $\rho$ is the probability that the second coordinates of the first and last vertex of $P_i(\omega)$ are actually equal.  In particular, we have that $\pi(S)\approx \frac 1 2$.

The vertex-boundary $\partial(S)$ of $S$ consists of states for which the second coordinates of the first and last vertex of $P_1$ differ by at most 1 (or: where one of the paths is empty: a linear-sized set).  For these states, the intersection of each color class with the `middle set' $\{2\}\times[\ell]$ must also be a path; see Figure \ref{fig:3xl}, middle.  There are only $O(\ell^2)$ such states, while the total number of states in $\Omega$ is greater than $2^\ell$ (when the top and bottom rows are constant with different colors, each of the middle vertices can be colored arbitrarily).  Thus $\pi(\partial S)=O(\frac{\ell^2}{2^\ell})$, and the chain has exponential mixing time.  (It is not hard to extend this argument to show that $G_{4\times \ell}$ has exponential mixing time, although understanding the general $G_{c\times \ell}$ for constant $c$ already seems like a challenge.)

The bottleneck we have identified for the color chain separates the chain into two regions, which are inverted by changing the labels of the two color classes; thus it does not show a bottleneck for the partition chain.

\bigskip 
To show a bottleneck for the partition chain, we define $G_{3\times \ell}^4$, for $\ell$ odd, to be the subgraph of the $\ell\times \ell$ grid graph induced by the set of vertices at graph distance $\leq 1$ from the middle horizontal or middle vertical. Some divisions of this graph (for $\ell=21)$ into two color classes are shown in Figure \ref{fig:^43xl}.  Note that the bandwidth of $G^4_{3\times \ell}$ is at most 12, which can be realized by a $\sigma$ which orders the vertices in, say, increasing distance from the central vertex of $G^4_{3\times \ell}$ (the unique vertex of the graph which is fixed by all of its symmetries).

\begin{figure}
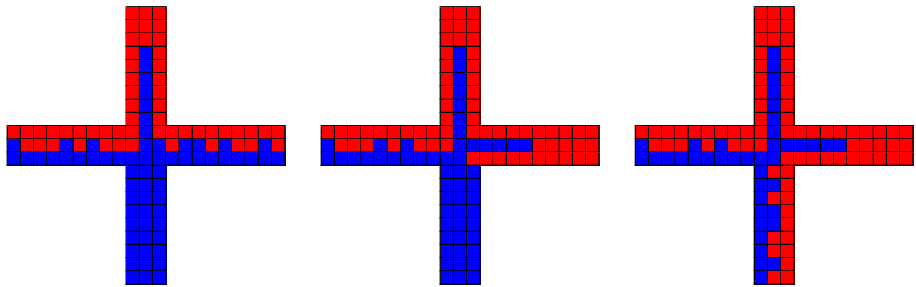

    \centering
\begin{minipage}{.33\linewidth}
\newdimen\omsq  \omsq=5pt
\newdimen\omrule    \omrule=.5pt
\newdimen\omint

\newif\ifvth    \newif\ifhth    \newif\ifomblank
\OMINO{
..................+-+-+-+..................\\
..................|r|r|r|..................\\
..................+-+-+-+..................\\
..................|r|r|r|..................\\
..................+-+-+-+..................\\
..................|r|r|r|..................\\
..................+-+-+-+..................\\
..................|r|b|r|..................\\
..................+-+-+-+..................\\
..................|r|b|r|..................\\
..................+-+-+-+..................\\
..................|r|b|r|..................\\
..................+-+-+-+..................\\
..................|r|b|r|..................\\
..................+-+-+-+..................\\
..................|r|b|r|..................\\
..................+-+-+-+..................\\
..................|r|b|r|..................\\
+-+-+-+-+-+-+-+-+-+-+-+-+-+-+-+-+-+-+-+-+-+\\
|r|r|r|r|r|r|r|r|r|r|b|r|r|r|r|r|r|r|r|r|r|\\
+-+-+-+-+-+-+-+-+-+-+-+-+-+-+-+-+-+-+-+-+-+\\
|b|r|r|r|b|r|b|r|r|r|b|b|r|b|b|r|b|r|r|b|r|\\
+-+-+-+-+-+-+-+-+-+-+-+-+-+-+-+-+-+-+-+-+-+\\
|b|b|b|b|b|b|b|b|b|b|b|b|b|b|b|b|b|b|b|b|b|\\
+-+-+-+-+-+-+-+-+-+-+-+-+-+-+-+-+-+-+-+-+-+\\
..................|b|b|b|..................\\
..................+-+-+-+..................\\
..................|b|b|b|..................\\
..................+-+-+-+..................\\
..................|b|b|b|..................\\
..................+-+-+-+..................\\
..................|b|b|b|..................\\
..................+-+-+-+..................\\
..................|b|b|b|..................\\
..................+-+-+-+..................\\
..................|b|b|b|..................\\
..................+-+-+-+..................\\
..................|b|b|b|..................\\
..................+-+-+-+..................\\
..................|b|b|b|..................\\
..................+-+-+-+..................\\
..................|b|b|b|..................\\
..................+-+-+-+..................\\
}
\end{minipage}%
\begin{minipage}{.33\linewidth}
\newdimen\omsq  \omsq=5pt
\newdimen\omrule    \omrule=.5pt
\newdimen\omint

\newif\ifvth    \newif\ifhth    \newif\ifomblank
\OMINO{
..................+-+-+-+..................\\
..................|r|r|r|..................\\
..................+-+-+-+..................\\
..................|r|r|r|..................\\
..................+-+-+-+..................\\
..................|r|r|r|..................\\
..................+-+-+-+..................\\
..................|r|b|r|..................\\
..................+-+-+-+..................\\
..................|r|b|r|..................\\
..................+-+-+-+..................\\
..................|r|b|r|..................\\
..................+-+-+-+..................\\
..................|r|b|r|..................\\
..................+-+-+-+..................\\
..................|r|b|r|..................\\
..................+-+-+-+..................\\
..................|r|b|r|..................\\
+-+-+-+-+-+-+-+-+-+-+-+-+-+-+-+-+-+-+-+-+-+\\
|r|r|r|r|r|r|r|r|r|r|b|r|r|r|r|r|r|r|r|r|r|\\
+-+-+-+-+-+-+-+-+-+-+-+-+-+-+-+-+-+-+-+-+-+\\
|b|r|r|r|b|r|b|r|r|r|b|b|b|b|b|b|r|r|r|r|r|\\
+-+-+-+-+-+-+-+-+-+-+-+-+-+-+-+-+-+-+-+-+-+\\
|b|b|b|b|b|b|b|b|b|b|b|r|r|r|r|r|r|r|r|r|r|\\
+-+-+-+-+-+-+-+-+-+-+-+-+-+-+-+-+-+-+-+-+-+\\
..................|b|b|b|..................\\
..................+-+-+-+..................\\
..................|b|b|b|..................\\
..................+-+-+-+..................\\
..................|b|b|b|..................\\
..................+-+-+-+..................\\
..................|b|b|b|..................\\
..................+-+-+-+..................\\
..................|b|b|b|..................\\
..................+-+-+-+..................\\
..................|b|b|b|..................\\
..................+-+-+-+..................\\
..................|b|b|b|..................\\
..................+-+-+-+..................\\
..................|b|b|b|..................\\
..................+-+-+-+..................\\
..................|b|b|b|..................\\
..................+-+-+-+..................\\
}
\end{minipage}%
\begin{minipage}{.33\linewidth}
\newdimen\omsq  \omsq=5pt
\newdimen\omrule    \omrule=.5pt
\newdimen\omint

\newif\ifvth    \newif\ifhth    \newif\ifomblank
\OMINO{
..................+-+-+-+..................\\
..................|r|r|r|..................\\
..................+-+-+-+..................\\
..................|r|r|r|..................\\
..................+-+-+-+..................\\
..................|r|r|r|..................\\
..................+-+-+-+..................\\
..................|r|b|r|..................\\
..................+-+-+-+..................\\
..................|r|b|r|..................\\
..................+-+-+-+..................\\
..................|r|b|r|..................\\
..................+-+-+-+..................\\
..................|r|b|r|..................\\
..................+-+-+-+..................\\
..................|r|b|r|..................\\
..................+-+-+-+..................\\
..................|r|b|r|..................\\
+-+-+-+-+-+-+-+-+-+-+-+-+-+-+-+-+-+-+-+-+-+\\
|r|r|r|r|r|r|r|r|r|r|b|r|r|r|r|r|r|r|r|r|r|\\
+-+-+-+-+-+-+-+-+-+-+-+-+-+-+-+-+-+-+-+-+-+\\
|b|r|r|r|b|r|b|r|r|r|b|b|b|b|b|b|r|r|r|r|r|\\
+-+-+-+-+-+-+-+-+-+-+-+-+-+-+-+-+-+-+-+-+-+\\
|b|b|b|b|b|b|b|b|b|b|b|r|r|r|r|r|r|r|r|r|r|\\
+-+-+-+-+-+-+-+-+-+-+-+-+-+-+-+-+-+-+-+-+-+\\
..................|b|r|r|..................\\
..................+-+-+-+..................\\
..................|b|b|r|..................\\
..................+-+-+-+..................\\
..................|b|r|r|..................\\
..................+-+-+-+..................\\
..................|b|b|r|..................\\
..................+-+-+-+..................\\
..................|b|b|r|..................\\
..................+-+-+-+..................\\
..................|b|r|r|..................\\
..................+-+-+-+..................\\
..................|b|r|r|..................\\
..................+-+-+-+..................\\
..................|b|b|r|..................\\
..................+-+-+-+..................\\
..................|b|r|r|..................\\
..................+-+-+-+..................\\
}
\end{minipage}
\caption{Partitions of $G=G^4_{3\times \ell}$ (for $\ell=21)$ into two connected color classes.  There are many more colorings which agree on the boundary vertices of $G$ with the first and third shown coloring than with the middle shown coloring, which forms part of a bottleneck in the chain.}
    \label{fig:^43xl}
\end{figure}

We again consider the outer face of the natural plane drawing of $G$, which is a $4\ell-2$ cycle, which we denote by $B^4_{3\times \ell}$.  This cycle has 4 vertices which have degree 4 in $G$ (the corner vertices of the central $3\times 3$ square of the graph); we denote these vertices by $b_1,b_2,b_3,b_4$.  They divide $B^4_{3\times \ell}$ into four paths; $Q_0$ from $b_4$ to $b_1,$, $Q_1$ from $b_1$ to $b_2$, $Q_2$ from $b_2$ to $b_3$, and $Q_3$ from $b_3$ to $b_4$.

Now for a coloring $\omega$, we let $\beta(\omega)$ be the set of $i \in \{1,2,3,4\}$, for which the interior of the path $Q_i$ includes vertices of both color.  In particular, letting $P_i=P_i(\omega)$ ($i=1,2$) denote the intersection of color class $i$ with $B^4_{3\times \ell}$ (which is necessarily a path), 
\[
\beta(\omega)=\{i\mid \mathring Q_i\cap P_1\neq \varnothing \text{ AND } \mathring Q_i\cap P_2\neq \varnothing\}.
\]
For example, if in Figure 3, $P_1$ is the ``North'' boundary segment, $P_2$ the ``East'' boundary segment, and $P_3$ and $P_4$ ``South'' and ``West'', respectively, then the coloring shown at left belongs has $\beta(\omega)=\{2,4\}$, the middle coloring has $\beta(\omega)=\{2\},$ and the coloring at right has $\beta(\omega)=\{2,3\}$.

\begin{observation}
For all $\omega$, we have $|\beta(\omega)|\leq 2$.
\end{observation}
\begin{proof}
If $i\in \beta(\omega)$, then the interior of $Q_i$ intersects both $P_1$ and $P_2$.  Since each of $P_1$ and $P_2$ has just two endpoints, and the interiors of the $Q_i$ are pairwise disjoint, there can be at most two such $i$.
\end{proof}

The key point now is that for any distinct $i,j\in \{1,2,3,4\}$, there are $\Omega(2^{\ell})$ colorings in $\beta^{-1}(\{i,j\})$; indeed, for the case where half of $P_1$ belongs to one color class and half to the other, we can freely choose the colors of the $\approx \ell/2$ ``middle vertices''---those that are adjacent to two vertices of this path.   Thus the cases where this is true for both $P_1$ and $P_2$ already gives $\Omega(2^\ell)$ members $\omega\in \beta^{-1}(\{i,j\})$.  On the other hand, letting $\beta_0$ be the set of colorings for which one of the paths $P_i$ is empty, the set
\[
\beta^{-1}(i)\cup \beta^{-1}(j)\cup \beta_0
\]
forms a cut set separating $\beta^{-1}(\{i,j\})$ from the rest of the Markov chain.  But this set has size $O(\mathrm{poly}(\ell)2^{\ell/2})$, showing that the mixing time of the chain is exponential in $\ell$.

\bibliographystyle{plain}
\bibliography{main}

\end{document}